\documentclass{article}
\usepackage{latexsym}
\usepackage{amsxtra}
\usepackage{amssymb}
\usepackage{amsthm}
\usepackage{xspace}

\usepackage{euscript}
\input xy
\xyoption{all} \CompileMatrices

\newtheorem{theorem}{Theorem}[section]
\newtheorem{proposition}[theorem]{Proposition}
\newtheorem{corollary}[theorem]{Corollary}
\newtheorem{lemma}[theorem]{Lemma}

\theoremstyle{definition}
\newtheorem{example}[theorem]{Example}

\newtheorem{problem}[theorem]{Problem}
\newtheorem{conjecture}[theorem]{Conjecture}

\begin{document}

\def\setmin{{-}}
\def\UU{{\mathcal U}}
\def\Z{{\mathcal Z}}
\def\PPP{{\mathcal P}}
\def\CCC{{\mathcal C}}
\def\XXX{{\mathcal X}}
\def\YYY{{\mathcal Y}}
\def\Mod{\operatorname{Mod}}
\def\Moore{\operatorname{Moore}}
\def\SStar{\operatorname{Semistar}}
\def\id{\operatorname{id}}
\def\Lin{\operatorname{Lin}}
\def\Spec{\operatorname{Spec}}
\def\Cl{\operatorname{Cl}}
\def\Idem{\operatorname{Idem}}
\def\Fix{\operatorname{Fix}}
\def\Cld{\operatorname{-Cl}}
\def\Inv{{\operatorname{Inv}}}
\def\Int{\operatorname{Int}}
\def\Pic{\operatorname{Pic}}
\def\SS{\operatorname{S}}
\def\U{{\textsf{\textup{U}}}}
\def\V{\operatorname{Inv}}
\def\Mor{\operatorname{Mor}}
\def\Int{{\textsf{\textup{Int}}}}
\def\I{{\textsf{\textup{I}}}}
\def\G{{\textsf{\textup{G}}}}
\def\K{{\textsf{\textup{K}}}}
\def\T{{\textsf{\textup{T}}}}
\def\R{{\textsf{\textup{R}}}}
\def\M{{\textsf{\textup{M}}}}
\def\FS{{\textsf{\textup{F}}}}
\def\Idl{{\textsf{\textup{Idl}}}}
\def\A{{\textsf{\textup{A}}}}
\def\C{{\textsf{\textup{C}}}}
\def\N{{\textsf{\textup{N}}}}
\def\Fp{{\mathcal F}}
\def\F{{\mathcal K}}
\def\S{\mathcal S}
\def\P{\operatorname{Prin}}
\def\B{\mathcal B}
\def\f{f}
\def\Inver{\operatorname{-Inv}}
\def\Prin{\textup{-}\mathcal{P}}
\def\Hom{\operatorname{Hom}}
\def\TM{t\textup{-Max}}
\def\TS{t\textup{-Spec}}
\def\Max{\operatorname{Max}}
\def\SSpec{\textup{-Spec}}
\def\SGV{\star\textup{-GV}}
\def\core{\operatorname{-core}}
\def\WBF{\operatorname{WBF}}
\def\sKr{\operatorname{sKr}}
\def\Loc{\operatorname{Loc}}
\def\Ass{\operatorname{Ass}}
\def\wAss{\operatorname{wAss}}
\def\ann{\operatorname{ann}}
\def\ind{\operatorname{ind}}
\def\cl{\star}
\def\ZZ{{\mathbb Z}}
\def\CC{{\mathbb C}}
\def\NN{{\mathbb N}}
\def\RR{{\mathbb R}}
\def\QQ{{\mathbb Q}}
\def\FF{{\mathbb F}}
\def\mm{{\mathfrak m}}
\def\nn{{\mathfrak n}}
\def\aaa{{\mathfrak a}}
\def\bbb{{\mathfrak b}}
\def\ppp{{\mathfrak p}}
\def\qqq{{\mathfrak q}}
\def\MM{{\mathfrak M}}
\def\qq{{\mathfrak Q}}
\def\rr{{\mathfrak R}}
\def\DD{{\mathfrak D}}
\def\cc{{\mathfrak S}}

\title{Semistar operations on Dedekind domains}
\date{\today} \author{Jesse Elliott \\ California State University, Channel Islands \\ One University Drive \\ Camarillo, CA 93012 USA \\ \texttt{jesse.elliott@csuci.edu}}

\maketitle

\begin{abstract}
We give an explicit description of the lattice $\SStar(D)$ of all semistar operations on any Dedekind domain $D$ from its set $\Max(D)$ of maximal ideals.  This descpription is constructive if $\Max(D)$ is finite.  As a corollary we show that $2^{{n \choose [n/2]}} \leq |\SStar(D)| \leq 2^{2^n}$  if $n = |\Max(D)|$ is finite; we compute $|\SStar(D)|$ if $|\Max(D)| \leq 7$; and we show that if $\Max(D)$ is infinite then $\SStar(D)$ has cardinality $2^{2^{|\Max(D)|}}$.

\ \\ 

\noindent {\bf Keywords:} semistar operation, Dedekind domain, principal ideal domain, closure operator, Moore family, multiplicative lattice

\noindent {\bf MSC:} 13F05, 13F10, 13A15, 13G05, 06A15, 06F05
\end{abstract}

\section{Introduction}

Let $D$ be an integral domain with quotient field $K$.  A {\it semistar operation on $D$} is a closure operation $\star$ on the lattice $\F(D)$ of all nonzero $D$-submodules of $K$ such that  $I^\star J^\star \subseteq (IJ)^\star$ for all $I,J \in \F(D)$ \cite{oka} \cite[Theorem 1.4]{ell}.  Semistar operations were introduced in \cite{oka} as a generalization of {\it star operations}, which were introduced by Krull in \cite[Section 6.43]{kru} in the guise of his {\it $'$-operations}.

An often-studied problem in the existing literature on semistar operations is to compute the cardinality of $\SStar(D)$.  It is clear that $D$ is a field if and only if $|\SStar(D)| = 1$, and it is well-known that $D$ is a discrete rank one valuation domain (DVR) if and only if $|\SStar(D)| = 2$ \cite[Theorem 48]{oka}. One also has the following.

\begin{theorem}\label{prufercounting}
Let $D$ be an integrally closed domain.
\begin{enumerate}
\item $\SStar(D)$ is finite if and only if $D$ is a finite dimensional Pr\"ufer domain such that the set $\Max(D)$ of its maximal ideals is finite \cite[(5.2)]{mat}. 
\item Suppose that $D$ is local and of finite dimension $n$.  Then $|\SStar(D)| \geq n+1$, with equality holding if and only if $D$ is a discrete valuation domain (of rank $n$); and $D$ is a valuation domain if and only if $|\SStar(D)| \leq 2n+1$ \cite[Theorems 3 and 4]{mat1}.
\end{enumerate}
\end{theorem} 

Note that the set $\SStar(D)$ of all semistar operations on $D$ is not just a set: it is naturally partially ordered by the relation $\leq$, where $\star_1 \leq \star_2$ if $I^{\star_1} \subseteq I^{\star_2}$ for all $I \in \F(D)$.  The partially ordered set (poset) $\SStar(D)$ is a complete lattice, and one has
$$I^{\bigwedge \Gamma} = \bigcap\{I^\star:  \star \in \Gamma\},$$ $$I^{\bigvee \Gamma} = \bigcap\{J \in \F(D): J \supseteq I \mbox{ and } \forall \star \in \Gamma \ (J^\star = J)\},$$ for all $\Gamma \subseteq \SStar(D)$ and all $I \in \F(D)$ \cite[Proposition 13.6]{ell}.   In Theorem \ref{dedthm} and Corollary \ref{dedcor1} we give an explicit description of the lattice $\SStar(D)$ of all semistar operations on any Dedekind domain $D$ in terms of the set $\Max(D)$ of its maximal ideals.

For $\Max(D)$ finite our description of the lattice $\SStar(D)$ is constructive and yields the following theorem.   Let $J \in \F(D)$.  For all $I \in \F(D)$ we let $I^{v(J)} = (J:_K (J:_K I))$.   The operation $v(J)$ on $\F(D)$ is a semistar operation on $D$ called {\it divisorial closure on $D$ with respect to $J$} \cite[Example 1.8(a)]{pic}.  Note that $v(D)$ is the {\it divisorial closure} semistar operation $v: I \longmapsto (I^{-1})^{-1}$ on $D$.  Also, for any set $X$ we let $2^X$ denote the power set of $X$, partially ordered by inclusion.  

\begin{theorem}\label{pid}
Let $D$ be a PID such that $\Max(D)$ is finite.
\begin{enumerate}
\item The lattice $\SStar(D)$ is isomorphic to the lattice $\C(2^{\Max(D)})$ of all closure operations on the complete lattice $2^{\Max(D)}$.  Explicitly, the map $$\C(2^{\Max(D)})   \longrightarrow  \SStar(D)$$ acting by
$$* \longmapsto \bigwedge\left\{v(I): I \in \F(D) \mbox{ and }\{\ppp \in \Max(D): I D_\ppp = K\} \in (2^{\Max(D)})^*\right\}$$
is a poset isomorphism; alternatively, one has $* \longmapsto \star$, where $\star$ is the largest semistar operation on $D$ such that $I \in \F(D)$ is $\star$-closed if $\{\ppp \in \Max(D): I D_\ppp = K\}$ is $*$-closed in $2^{\Max(D)}$.
\item The lattice $\SStar(D)$ is anti-isomorphic to the subposet $$\Moore(2^{\Max(D)}) = \left\{\YYY \subseteq 2^{\Max(D)} : \YYY \mbox{ is closed under arbitrary intersections}\right\}$$ of $2^{2^{\Max(D)}}$.  
Explicitly, the map $$\SStar(D) \longrightarrow \Moore(2^{\Max(D)})$$ acting by
$$\star \longmapsto \{ \{\ppp \in \Max(D): I^\star D_\ppp = K\}: I \in \F(D) \} $$
is a poset anti-isomorphism with inverse acting by
$$\YYY \longmapsto \bigwedge\left\{v(I): I \in \F(D) \mbox{ and }\{\ppp \in \Max(D): I D_\ppp = K\} \in \YYY\right\}.$$
\end{enumerate}
\end{theorem}

In Section \ref{sec:con} we obtain the following as a corollary of these constructions.



\begin{theorem}\label{mainthm} Let $D$ and $D'$ be Dedekind domains.
\begin{enumerate}
\item The lattices $\SStar(D)$ and $\SStar(D')$  are isomorphic if $\Max(D)$ and $\Max(D')$ have the same cardinality. 
\item  If $\Max(D)$ is infinite, then $|\SStar(D)|$ is equal to $2^{2^{|\Max(D)|}}$.   
\item  If $|\Max(D)| = n$ is finite, then $2^{{n \choose [n/2]}} \leq |\SStar(D)| \leq 2^{2^n}$, and $|\SStar(D)|$ is given for $n \leq 7$ as in Table 1.
\item If $|\Max(D)| = 2$, then $\SStar(D)$ is isomorphic to $2^{\{1,2,3\}}\setmin\{\{1\}\}$.
\item If $|\Max(D)| = 3$, then $\SStar(D)$ is anti-isomorphic to the lattice with Hasse diagram given as in \cite[Figure 1]{hig}.
\end{enumerate}
\end{theorem}

\begin{table}
\caption{Number of semistar operations on a Dedekind domain $D$}
\centering 
\begin{tabular}{c|c}
$|\Max(D)|$ & $|\SStar(D)|$ \\ \hline
$1$ & $2$ \\ 
$2$ & $7$ \\
$3$ & $61$ \\
$4$ & $2\ 480$ \\
$5$ & $1\ 385\ 552$ \\
$6$ & $75\ 973\ 751\ 474$ \\
$7$ & $14\ 087\ 648\ 235\ 707\ 352\ 472$
\end{tabular}
\end{table}

In the spirit of Theorem \ref{prufercounting} we make the following conjecture.

\begin{conjecture}
Let $D$ be an integrally closed domain such that $\Max(D)$ is finite.  Then $D$ is a principal ideal domain (PID) if and only if $|\SStar(D)|$ is equal to the number of subsets of $2^{\Max(D)}$ that are closed under arbitrary intersections.
\end{conjecture}

\section{Background on semistar operations and nuclei}


A {\it magma} is a set $M$ equipped with a binary operation on $M$ (which we write multiplicatively).  An {\it ordered magma} is a magma $M$ equipped with a partial ordering $\leq$ such that $x \leq y$ and $x' \leq y'$ imply $xx' \leq yy'$ for all $x, y, x', y' \in M$.    An ordered magma $M$ is said to be a {\it multiplicative lattice} (resp., {\it near multiplicative lattice}) if $M$ is a commutative monoid and one has
$\bigvee (XY) = (\bigvee X)(\bigvee Y)$ for all subsets (resp., nonempty subsets) $X,Y$ of $M$ \cite{ell}.  A multiplicative lattice is equivalently a near multiplicative lattice $M$ such that $\bigwedge M$ exists and annihilates every element of $M$.  For the purposes of this paper the relevant examples are as follows.

\begin{example} Let $D$ be an integral domain with quotient field $K$.
\begin{enumerate}
\item The complete lattice $\Mod_D(K)$ of all $D$-submodules of $K$ is a multiplicative lattice under the operation $(I,J) \longmapsto IJ$.
\item The lattice $\F(D)$ of all nonzero $D$-submodules of $K$ is a near multiplicative lattice.
\end{enumerate}
\end{example}

For any self-map $\star$ of a set $X$ we write $x^\star = \star(x)$ for all $x \in X$.  
A {\it closure operation} on a poset $S$ is a self-map $\star$ of $S$ satisfying the following conditions for all $x,y \in S$.
\begin{enumerate}
\item $x \leq x^\cl$.
\item $x \leq y$ implies $x^\cl \leq y^\cl$.
\item $(x^\cl)^\cl = x^\cl$.
\end{enumerate}
 A {\it nucleus} on an ordered magma $M$ is a closure operation $\star$ on $M$ such that $x^\star y^\star \leq (xy)^\star$ for all $x,y \in M$.   Nuclei were first studied in the contexts of ideal lattices and locales and later in the context of quantales as {\it quantic nuclei} \cite{nie}.

\begin{theorem}[{\cite[Theorem 1.4]{ell}}]
A semistar operation on an integral domain $D$ is equivalently a nucleus on the near multiplicative lattice $\F(D)$.  Moreover, the following are equivalent for any self-map $\star$ of $\F(D)$.
\begin{enumerate}
\item $\star$ is a semistar operation on $D$.
\item $\star$ is a closure operation on the lattice $\F(D)$ and $(aI)^\star = aI^\star$ for all $I \in \F(D)$ and all nonzero $a \in D$.
\item $\star$ is a closure operation on the lattice $\F(D)$ and $(I^\star J^\star)^\star = (IJ)^\star$ for all $I,J \in \F(D)$.
\item $IJ \subseteq K^\star$ if and only if $I J^\star \subseteq K^\star$ for all $I,J,K \in \F(D)$.
\end{enumerate}
\end{theorem}

If $\star_1$ and $\star_2$ are nuclei on an ordered magma $M$, then we write $\star_1 \leq \star_2$ if $x^{\star_1} \leq  x^{\star_2}$ for all $x \in M$.  This defines a partial ordering on the set $\N(M)$ of all nuclei on $M$. If $\bigvee X$ exists in $M$ for all nonempty subsets $X$ of $M$ (which holds if $M$ is a near multiplicative lattice), then by \cite[Proposition 4.3]{ell} the poset $\N(M)$ is a complete lattice and one has $$x^{\bigwedge\Gamma} = \bigwedge\{x^\star: \star \in \Gamma\},$$
$$x^{\bigvee\Gamma} = \bigwedge\{y \in M: y \geq x \mbox{ and } \forall \star \in \Gamma \ (y^\star = y)\},$$ for all $\Gamma \subseteq \N(M)$ and all $x \in M$.

\begin{corollary}
For any integral domain $D$ one has $\SStar(D) = \N(\F(D))$ as partially ordered sets.
\end{corollary}

If $M$ is an ordered magma, then the map $\N(M) \longrightarrow 2^M$ acting by $\star \longmapsto M^\star$ is an anti-embedding of $\N(M)$ in $2^M$.  We will denote the image $\{M^\star: \star \in \N(M)\}$ of this anti-embedding by $\M(M)$.  In particular, the poset $\N(M)$ is anti-isomorphic to the poset $\M(M)$.  \cite[Proposition 4.6]{ell} implies the following.

\begin{proposition}\label{mprop}
Let $D$ be an integral domain with quotient field $K$.  The poset $\M(\F(D))$ consists of all sets of the form $\CCC \setmin\{(0)\}$, where $\CCC$ is any subset of $\F(D)$ satisfying the following three condtions: (1) $(0) \in \CCC$; (2) if $\mathcal{I} \subseteq \CCC$ then $\bigcap \mathcal{I} \in \CCC$; and (3) if $I \in {\mathcal C}$ then $(I:_K J) \in \CCC$ for all $J \in \F(D)$.  Moreover, the complete lattice $\SStar(D) = \N(\F(D))$ of all semistar operations on $D$  is anti-isomorphic to the poset $\M(\F(D))$.
\end{proposition}

\section{Constructing the lattice of all semistar operations}\label{sec:SO}

In this section we construct the lattice $\SStar(D) = \N(\F(D))$ of all semistar operations on any Dedekind domain $D$ from the set $\Max(D)$ of its maximal ideals.  

If $D$ is a Dedekind domain, then for any $\ppp \in \Max(D)$ we let $v_\ppp: K \longrightarrow \ZZ \cup\{\infty\}$ denote the normalized $\ppp$-adic valuation on $K$.  We let $\ZZ[\infty]$ denote the near multiplicative lattice $\ZZ \cup \{\infty\}$, where $a \leq \infty$ for all $a \in \ZZ$ and $a + \infty = \infty = \infty + a$ for all $a \in \ZZ[\infty]$.  Also, we let $\ZZ[\pm \infty]$ denote the multiplicative lattice obtained from $\ZZ[\infty]$ by adjoining an annihilator $-\infty = \inf \ZZ[\infty]$.  In particular, we assume $-\infty + \infty = -\infty = \infty + (-\infty)$.  Also, for any $x \in \ZZ[\pm \infty]$ we define $-x$ and $|x|$ in $\ZZ[\pm \infty]$ in the obvious way.

For any near multiplicative lattice $M$ and any set $S$ the set $M^S$ of all functions from $S$ to $M$ is a near multiplicative lattice under pointwise multiplication.  For any set $S$ we let
$$\Z(S) = \left\{f: S \longrightarrow \ZZ[\infty]: f^{-1}(\ZZ_{<0}) \mbox{ is finite}\right\},$$
which is a sub near multiplicative lattice of the multiplicative lattice ${\ZZ[\pm\infty]}^{S}$.

In the next proposition we show that $\F(D) \cong \Z(\Max(D))$ as near multiplicative lattices.  Let $\ppp \in \Max(D)$.  Define $\ppp^\infty D_\ppp = (0)$ and $\ppp^{-\infty}D_\ppp = K$.  For any $I \in \F(D)$ we define $v_\ppp(I) \in \ZZ\cup \{\pm \infty\}$ so that $ID_\ppp = \ppp^{v_\ppp(I)}D_\ppp$.  For any function
$f: \Max(D) \longrightarrow \ZZ[\pm \infty]$ we let $$[f] = \bigcap_{\ppp \in \Max(D)} \ppp^{-f(\ppp)} D_\ppp = \{x \in K: v_\ppp(x) \geq -f(\ppp) \mbox{ for all } \ppp \in \Max(D)\}.$$
There is a map $[-]: {\ZZ[\pm \infty]} ^{\Max(D)} \longrightarrow \operatorname{Mod}_D(K)$ acting by $f \longmapsto [f]$.

\begin{proposition}\label{dedlemma}
Let $D$ be a Dedekind domain, and let $f, g \in \ZZ[\pm \infty]^{\Max(D)}$.
\begin{enumerate}
\item $f \in \Z(\Max(D))$ if and only if $[f] \neq (0)$.
\item  If $f \in \Z(\Max(D))$, then $[f]D_\ppp = \ppp^{-f(\ppp)} D_\ppp$ for all $\ppp \in \Max(D)$.
\item If $f,g \in \Z(\Max(D))$, then $([f]:_K [g]) = [-(-f+g)]$.
\item The map $[-]: \Z(\Max(D)) \longrightarrow \F(D)$ is an isomorphism of near multiplicative lattices with inverse acting by $I \longmapsto (\ppp \longmapsto -v_\ppp(I))$.
\end{enumerate}
\end{proposition}

\begin{proof}
Statement (1) is clear. 
We prove statement (2).  Clearly one has $[f]D_\ppp \subseteq \ppp^{-f(\ppp)}D_\ppp$.  We wish to show the reverse containment.  Let $x \in \ppp^{-f(\ppp)}D_\ppp$.   By the weak approximation theorem for Dedekind domains, there exists $y \in K$ such that
$$v_\qqq(y) = \left\{\begin{tabular}{ll} $v_{\qqq}(x)$ & if $\qqq = \ppp$  \\
		$-f(\qqq)$ & if $\qqq \neq \ppp$ and $f(\qqq) < 0$ \\	
		$\geq 0$ & otherwise.
 \end{tabular}\right.$$
Thus $xD_\ppp = yD_\ppp$ and $y \in [f]$, whence $x \in [f]D_\ppp$. This proves (2).
To prove statement (3), note that
\begin{eqnarray*}
([f]:_K [g]) & = & \bigcap_\ppp \left(\ppp^{-f(\ppp)} D_\ppp :_K [g]\right) \\
  & = & \bigcap_{f(\ppp) < \infty} \left(\ppp^{-f(\ppp)} \left(D_\ppp :_K [g]D_\ppp  \right)\right) \\
  & = & \bigcap_{f(\ppp) < \infty} \left(\ppp^{-f(\ppp)} (D_\ppp :_K \ppp^{-g(\ppp)} D_\ppp) \right) \\
  & = & \bigcap_{f(\ppp) < \infty} \ppp^{-f(\ppp)+g(\ppp)} D_\ppp \\
  & = & [-(-f+g)].
\end{eqnarray*}
Finally,  statement (4) follows readily from statement (3) and the fact that $I = \bigcap_{\ppp \in \Max(D)} \ppp^{v_\ppp(I)} D_\ppp$ and $v_\ppp(IJ) = v_\ppp(I)+v_\ppp(J)$ for all $I,J \in \F(D)$.
\end{proof}

For any set $S$ and any $f,g \in \ZZ[\pm\infty]^S$, we write $f \preceq g$ if $f(x) \leq g(x)$ for almost all $x \in S$ and $f(x) = \infty$ if and only if $g(x) = \infty$ for all $x \in S$. The relation $\preceq$ is a preorder on $\ZZ[\pm \infty]^S$.

\begin{lemma}\label{lem1}
Let $S$ be a set and $\XXX \subseteq \Z(S)$.  Consider the following conditions on $\XXX$.
\begin{enumerate}
\item If $\YYY \subseteq \XXX$ and $\inf \YYY \in \Z(S)$, then $\inf \YYY \in \XXX$.
\item If $f \in \XXX$ and $g \in \Z(S)$ with $g \preceq f$, then $g \in \XXX$.
\item If $f \in \XXX$ and $g, -(-f+g) \in \Z(S)$ then $-(-f+g) \in \XXX$.
\end{enumerate}
One has $(1) \wedge (2) \Leftrightarrow (1) \wedge (3)$ and $(2) \Rightarrow (3)$.  Moreover, if $S = \Max(D)$, where $D$ is a Dedekind domain, and if $[\XXX]$ for any $\XXX \subseteq \Z(\Max(D))$ denotes the image of $\XXX$ under the isomorphism $[-]: \Z(\Max(D)) \longrightarrow \F(D)$, then the above conditions hold if and only if $[\XXX] = \F(D)^\star$ for some semistar operation $\star$ on $D$.
\end{lemma}

\begin{proof}
First, suppose that condition (2) holds.  Let $f \in \XXX$ and $g \in \Z(S)$ with $h = -(-f+g) \in \Z(S)$.  Note that $h(x) = \infty$ if $f(x) = \infty$, and if $f(x) < \infty$ then $g(x) < \infty$ and $h(x) = f(x)-g(x) < \infty$.   In particular, $h(x) > f(x)$ implies $g(x) < 0$, which holds for only finitely many $x$.  It follows, then, that $h \preceq f$, whence $h = -(-f+g) \in \Z(S)$.  This shows that $(2) \Rightarrow (3)$.

Next, suppose $(1) \wedge (3)$ holds.  Let $f \in \XXX$ and $g \in \Z(S)$ with $g \preceq f$.  Let $x \in S$ with $g(x) < \infty$ (so $f(x) < \infty$).  Define
$h_x \in \Z(S)$ as follows:
$$h_x(y) = \left\{\begin{tabular}{ll} $f(y)-g(y)$ & if $y = x$ or if $y \neq x$ and $g(y) > f(y)$ \\
		$0$ & otherwise.
 \end{tabular}\right.$$
Then
$$-(-f+h_x)(y) = \left\{\begin{tabular}{ll} $g(y)$ & if $y = x$ or if $y \neq x$ and $g(y) > f(y)$ \\
		$f(y)$ & otherwise.
 \end{tabular}\right.$$
By condition (3) we have $-(-f + h_x) \in \XXX$.
Therefore, by condition (1), we have $g = \inf_{g(x) < \infty} -(-f + h_x) \in \XXX$.
Thus $(1) \wedge (3) \Rightarrow (2)$, and since $(2) \Rightarrow (3)$ we have $(1) \wedge (2) \Leftrightarrow (1) \wedge (3)$.

Finally, if $S = \Max(D)$, where $D$ is a Dedekind domain, then, by Propositions \ref{mprop} and \ref{dedlemma}, one has  $[\XXX] = \F(D)^\star$ for some semistar operation $\star$ on $D$ if and only if conditions $(1)$ and $(3)$ hold.
\end{proof}

\begin{lemma}\label{lem2}
Let $S$ be a set.  For any $\XXX \subseteq \Z(S)$, let
$$\XXX^\preceq = \{g \in \Z(S): g \preceq f \mbox{ for some } f \in \XXX\},$$
and let
$$\XXX^{\inf} = \{\inf \YYY: \YYY \subseteq \XXX \mbox{ and } \inf \YYY \in \Z(S)\}.$$
\begin{enumerate}
\item $\XXX \longmapsto \XXX^\preceq$ and $\XXX \longmapsto \XXX^{\inf}$ are closure operations on $2^{\Z(S)}$.
\item For any $\XXX \subseteq \Z(S)$ one has $(\XXX^\preceq)^{\inf} \supseteq (\XXX^{\inf})^\preceq$.
\item $\XXX^\dagger = (\XXX^\preceq)^{\inf}$ is the smallest subset $\mathcal{W}$ of $\Z(S)$ containing $\XXX$ such that $\mathcal{W}^\preceq \subseteq \mathcal{W}$ and $\mathcal{W}^{\inf} \subseteq \mathcal{W}$.
\item $\XXX \longmapsto \XXX^\dagger$ is a closure operation on $2^{\Z(S)}$.
\end{enumerate}
\end{lemma}

\begin{proof}
Statement (1) is clear.  To prove (2), let $g \in (\XXX^{\inf})^\preceq$.  Then $g \preceq \inf_\lambda f_\lambda$, where $f_\lambda \in \XXX$ for all $\lambda$ in some indexing set and $\inf_\lambda f_\lambda \in \Z(S)$.  Now, for almost all $x \in S$ one has $g(x) \leq f_\lambda(x)$ for all $\lambda$, and also $g(x) = \infty$ if and only if $f_\lambda(x) = \infty$ for all $\lambda$.  Let
$$g_\lambda(x) = \left\{\begin{tabular}{ll} $g(x)$ & if $f_\lambda(x) < \infty$ \\
		$\infty$ & otherwise.
 \end{tabular}\right.$$
Since $g \in \Z(S)$ we have $g_\lambda \in \Z(S)$ for all $\lambda$.  Moreover, one 
has $g_\lambda \preceq f_\lambda$, and therefore $g_\lambda \in \XXX^\preceq$, for all $\lambda$.  Finally, since $g = \inf_\lambda g_\lambda$, we have $g \in (\XXX^\preceq)^{\inf}$.  This proves (2).  Finally, statements (3) and (4) follow readily from statement (2) and \cite[Proposition 4.8]{ell}.
\end{proof}

Combining Proposition \ref{dedlemma} with Lemmas \ref{lem1} and \ref{lem2} we obtain the following.

\begin{theorem}\label{dedthm}
Let $D$ be a Dedekind domain.   For any $\XXX \subseteq \Z(\Max(D))$, let
$$\XXX^\preceq = \{g \in \Z(\Max(D)): g \preceq f \mbox{ for some } f \in \XXX\},$$
and let
$$\XXX^{\inf} = \{\inf \YYY: \YYY \subseteq \XXX \mbox{ and } \inf \YYY \in \Z(\Max(D))\}.$$
Let $[\XXX]$ denote the image of $\XXX$ under the isomorphism $[-]: \Z(\Max(D)) \longrightarrow \F(D)$ of Proposition \ref{dedlemma}.
\begin{enumerate}
\item $\XXX^\dagger = (\XXX^\preceq)^{\inf}$ is the smallest subset of $\Z(\Max(D))$ containing $\XXX$ such that
$\left[\XXX^\dagger\right] = {\F(D)}^\star$ for some semistar operation $\star$ on $D$.
\item The self-map of $\dagger$ of $2^{\Z(\Max(D))}$ acting by $\XXX \longmapsto \XXX^\dagger$ is a closure operation on $2^{\Z(\Max(D))}$.
\item$\XXX \in  (2^{\Z(\Max(D))})^\dagger$ if and only if $\XXX^\preceq \subseteq \XXX$ and $\XXX^{\inf} \subseteq \XXX$.
\item The map $(2^{\Z(\Max(D))})^\dagger \longrightarrow \M(\F(D))$ acting by $\XXX \longmapsto [\XXX]$ is an isomorphism of posets; explicitly, this isomorphism acts by $$\XXX \longmapsto \left\{\bigcap_{\ppp \in \Max(D)} \ppp^{-f(\ppp)} D_\ppp : f \in \XXX\right\}$$
with inverse acting by
$$\mathcal{C} \longmapsto \{\ppp \longmapsto -v_\ppp(I): I \in \mathcal{C}\}.$$
\end{enumerate}
\end{theorem}

By \cite[Proposition 13.8]{ell} the inverse of the poset anti-isomorphism $\SStar(D) \longrightarrow \M(\F(D))$ acting by $\star \longmapsto \F(D)^\star$ acts by $\mathcal{C} \longmapsto \bigwedge \{v(J): J \in \mathcal{C}\}$.  Thus we have the following.

\begin{corollary}\label{dedcor1}
For any Dedekind domain $D$, the poset $\SStar(D)$ of all semistar operations on $D$ is anti-isomorphic to the subposet
\begin{eqnarray*}
(2^{\Z(\Max(D))})^\dagger   & = & \left\{ (\XXX^{\preceq})^{\inf} : \XXX  \in 2^{\Z(\Max(D))} \right\}\\
& =  &\left\{\XXX \in 2^{\Z(\Max(D))}: \XXX^{\preceq} \subseteq \XXX \mbox{ and } \XXX^{\inf} \subseteq \XXX\right\} 
\end{eqnarray*}
of $2^{\Z(\Max(D))}$.  Explicitly, there is a poset anti-isomorphism
$$\SStar(D) \longrightarrow (2^{\Z(\Max(D))})^\dagger$$ acting by
$$\star \longmapsto \{\ppp \longmapsto -v_\ppp(I): I \in \F(D)^\star\}$$
 with inverse acting by
$$\XXX \longmapsto \bigwedge\left\{v\left(\bigcap_{\ppp \in \Max(D)} \ppp^{-f(\ppp)} D_\ppp \right): f \in \XXX\right\}.$$
\end{corollary}

\section{Consequences of the construction}\label{sec:con}

The results in this section yield Theorem \ref{pid} and \ref{mainthm} of the introduction.

Let $X$ be a poset.  The set $\C(X)$ of closure operations on $X$ is partially ordered by the obvious relation.   The map $\C(X) \longrightarrow 2^X$ acting by $\star \longmapsto X^\star$ is an anti-embedding of $\C(X)$ in $2^X$.  We will denote the image $\{X^\star: \star \in \C(X)\}$ of this anti-embedding by $\Moore(X)$.  In particular, the poset $\C(X)$ is anti-isomorphic to the poset $\Moore(X)$.  

Let $S$ be a set.  By \cite[Corollary 4.7(1)]{ell}, one has
$$\Moore(2^S) = \{\YYY \subseteq 2^S : \YYY \mbox{ is closed under arbitrary intersections}\}.$$  ($\Moore(2^S)$ is the set of all {\it Moore families on $S$} \cite{col}.) For any $X \subseteq S$, define $\iota_X \in \Z(S)$ as follows:
$$\iota_X(x) = \left\{\begin{tabular}{ll} $\infty$ & if $x \in X$ \\
		$0$ & otherwise.
 \end{tabular}\right.$$
Then $$\{\iota_X\}^\dagger = \{f \in \Z(S): f(x) = \infty \mbox{ iff } x \in X, \mbox{ and } f(x) = 0 \mbox{ for almost all } x \in S\setmin X\}.$$  We let $\iota(X) = \{\iota_X\}^\dagger$, and for any $\YYY \subseteq 2^{S}$ we let 
\begin{eqnarray*}
\iota(\YYY) & = & \bigcup_{X \in \YYY} \iota(X) \\
  & = & \{f \in \Z(S): f^{-1}(\infty) \in \YYY \mbox{ and } f(x) = 0 \mbox{ for almost all } x \in S\setmin f^{-1}(\infty)\}.
\end{eqnarray*}

\begin{lemma}\label{dedembed}
Let $S$ be a set.
\begin{enumerate}
\item The map $\iota: 2^{2^S} \longrightarrow 2^{\Z(S)}$ is an embedding of posets and has an order-preserving left inverse acting by $\XXX \longmapsto \{f^{-1}(\infty): f \in \XXX\}$.
\item The map $\iota: \Moore(2^S) \longrightarrow (2^{\Z(S)})^\dagger$ is an embedding of posets and has an order-preserving left inverse acting by $\XXX \longmapsto \{f^{-1}(\infty): f \in \XXX\}$.
\item  If $S$ is finite, then the map $\iota: \Moore(2^S) \longrightarrow (2^{\Z(S)})^\dagger$ is a poset isomorphism, and one has $\iota(\YYY) = \{f \in \Z(S): f^{-1}(\infty) \in \YYY\}$ for all $\YYY \in \Moore(2^{S})$.
\end{enumerate}
\end{lemma}

\begin{proof}
Statement (1) is clear.  To prove (2), first note that if $f \in \iota(X)$ and $g \preceq f$, where $X \subseteq S$ and $g \in \Z(S)$, then $g \in \iota(X)$; and if $f_\lambda \in \iota(X_\lambda)$ for all $\lambda$ in some indexing set, where each $X_\lambda \subseteq S$ and $\inf_\lambda f_\lambda \in \Z(S)$, then $\inf_\lambda f_\lambda \in \iota
\left(\bigcap_\lambda X_\lambda \right)$.  This shows that $\iota$ maps $\Moore(2^S)$ into $(2^{\Z(S)})^\dagger$.  Next, if $\XXX \in (2^{\Z(S)})^\dagger$ and $f \in \XXX$, then $\iota_{f^{-1}(\infty)} \preceq f$, whence $\iota_{f^{-1}(\infty)} \in \XXX$.  Therefore, if $\YYY \subseteq \XXX$, then $\inf_{f \in \YYY} \iota_{f^{-1}(\infty)} = \iota_Y \in \XXX$, where $Y = {\bigcap_{f \in \YYY} f^{-1}(\infty)}$, so ${\bigcap_{f \in \YYY} f^{-1}(\infty)} \in \{f^{-1}(\infty): f \in \XXX\}$.  This shows that the left inverse of $\iota$ maps $(2^{\Z(S)})^\dagger$ into $\Moore(2^S)$ and therefore completes the proof of statement (2).  Finally, suppose that $S$ is finite.  Then for any $f \in \Z(S)$ one has $f \in \iota(f^{-1}(\infty))$.  Therefore, if $\XXX \in (2^{\Z(S)})^\dagger$, then $\XXX = \bigcup_{f \in \XXX} \iota(f^{-1}(\infty)) = \iota(\{f^{-1}(\infty): f \in \XXX\})$, which proves that the two maps in statement (2) are inverses of each other in this case. 
\end{proof}

Combining Lemma \ref{dedembed} above with Theorem \ref{dedthm} and Corollary \ref{dedcor1}, we obtain Theorem \ref{pid} of the introduction.  We also have the following.

\begin{corollary}\label{dedcard}
Let $D$ be a Dedekind domain.   If $\Max(D)$ is infinite, then the set $\SStar(D)$ of all semistar operations on $D$ has cardinality $2^{2^{|\Max(D)|}}$.
\end{corollary}

\begin{proof}
Note first that $\Z(\Max(D))$ has cardinality $\aleph_0^{|\Max(D)|} = 2^{|\Max(D)|}$.   Therefore, by Corollary \ref{dedcor1}, the set $\SStar(D)$ has cardinality at most $2^{2^{|\Max(D)|}}$.
On the other hand, by Theorem \ref{dedthm}(4) and Lemma \ref{dedembed}(2), the set $\SStar(D)$ has cardinality at least $|\Moore(2^{\Max(D)})|$.  Now, $\Moore(2^S)$ for any set $S$ is isomorphic to the poset
$$\UU(2^S) = \left\{\YYY \subseteq 2^S : \YYY \mbox{ is closed under arbitrary unions}\right\}.$$
Moreover, the poset $\UU(2^S)$ contains the set of all ultrafilters on the set $S$, which for any infinite set $S$ is known to have cardinality $2^{2^{|S|}}$.  Therefore $|\SStar(D)| \geq |\Moore(2^{\Max(D)})| = |\UU(2^{\Max(D)})| \geq 2^{2^{|\Max(D)|}}$, whence equalities hold.
\end{proof}

\begin{corollary}
The set of all semistar operations on $\ZZ$ has cardinality $2^{2^{\aleph_0}}$.
\end{corollary}


Given also \cite[Table 1]{col} and \cite[Figure 1]{hig}, we obtain the following.

\begin{corollary}\label{pid2}
Let $D$ be a PID such that $\Max(D)$ is finite.
\begin{enumerate}
\item $\SStar(D)$ is finite and has cardinality given as in Table 1 for small values of $|\Max(D)|$.
\item If $|\Max(D)| = 2$, then $\SStar(D) \cong 2^{\{1,2,3\}}\setmin\{\{1\}\}$.
\item If $|\Max(D)| = 3$, then $\SStar(D)$ is anti-isomorphic to the lattice with Hasse diagram given as in \cite[Figure 1]{hig}.
\end{enumerate}
\end{corollary}

It is known that, if $X_n$ is a set of cardinality $n$ for  every positive integer $n$, then 
$$|\Moore(2^{X_n})| = 2^{{n \choose [n/2]}f(n)},$$
where $f(n) \geq 1$ for all $n$ and $f(n) = 1 + O(n^{-1/4} \log_2 n)$ as $n \longrightarrow \infty$ and therefore $\lim_{n \rightarrow \infty} f(n) = 1$ \cite[Theorem 1]{Ale}.  Thus we have the following.

\begin{corollary}
Let $D$ be a Dedekind domain with quotient field $K$, and let $K \supsetneq D_1 \supsetneq D_2 \supsetneq D_3 \supsetneq \cdots$ be an infinite descending chain of overrings of $D$ such that $m_n = |\Max(D_n)| < \infty$ for all $n$.  Then we have the following.
\begin{enumerate}
\item There is a unique strictly ascending sequence $X_1 \subsetneq X_2 \subsetneq X_2 \subsetneq \cdots$ of subsets of $\Max(D)$ such that $D_n = U_n^{-1}D$, where $U_n = D \setmin \bigcup X_n$.
\item One has $m_n = |X_n|$ for all $n$, and therefore $m_n$ is a strictly increasing sequence of positive integers.
\item If $s_n = |\SStar(D_n)|$ for all $n$, then $s_n = 2^{{m_n \choose [m_n/2]}f(n)}$, where $f(n) \geq 1$ for all $n$ and $f(n)  = 1 + O(m_n^{-1/4} \log_2 m_n)$ as $n \longrightarrow \infty$, and therefore $\log_2 s_n \sim {m_n \choose [m_n/2]}$ as $n \longrightarrow \infty$.
\end{enumerate}
\end{corollary}

Note next that the poset $(2^{\Z(\Max(D))})^\dagger \cong \SStar(D)$ is constructed solely from the set $\Max(D)$.  Therefore, if $\Max(D)$ and $\Max(D')$ have the same cardinality for two Dedekind domains $D$ and $D'$, then the lattices $\SStar(D)$ and $\SStar(D')$ are isomorphic.   Regarding the converse, at least we can say this.  Consider the following statement.
\begin{enumerate}
\item[(P)] $|2^X| = |2^Y|$ implies $|X| = |Y|$ for all sets $X$ and $Y$.
\end{enumerate}
It is well-known that ZFC $\not \vdash$ P (if ZFC is consistent) and ZFC+GCH $\vdash$ P, where ZFC is Zermelo-Fraenkel set theory with the Axiom of Choice, and GCH is the Generalized Continuum Hypothesis (which in ZFC is equivalent to the statement that $Y \subseteq 2^X$ implies $|Y| \leq |X|$ or $|Y| = |2^X|$ for all sets $X$ and $Y$).  We have the following.

\begin{corollary}\label{equiv}
Consider the following conditions on Dedekind domains $D$ and $D'$.
\begin{enumerate}
\item $\Max(D)$ and $\Max(D')$ have the same cardinality.
\item The near multiplicative lattices $\F(D)$ and $\F(D')$ are isomorphic.
\item The lattices $\SStar(D)$ and $\SStar(D')$ are isomorphic.
\item  $\SStar(D)$ and $\SStar(D')$ have the same cardinality.
\end{enumerate}
One has $(1) \Leftrightarrow (2) \Rightarrow (3) \Rightarrow (4)$, and the four conditions are equivalent if either $\Max(D)$ or $\Max(D')$ is finite or if one assumes along with the axioms of ZFC the following statement: (P) $|2^X| = |2^Y|$ implies $|X| = |Y|$ for all sets $X$ and $Y$.  Moreover, the implication $(4) \Rightarrow (1)$ is equivalent in ZFC to the statement (P) and therefore can neither be proved nor disproved from ZFC if ZFC is consistent.
\end{corollary}

\begin{proof}
By Proposition \ref{dedlemma}(4) and \cite[Corollary 13.2]{ell} we have $(1) \Rightarrow (2) \Rightarrow (3) \Rightarrow (4)$.  Moreover, the elements of $\Max(D)$ can be recovered from the near multiplicative lattice $\F(D)$ as the elements lying immediately below the identity element $D$ of $\F(D)$.  Thus we have $(2) \Rightarrow (1)$.  Suppose, then, that condition $(4)$ holds.  If either $\Max(D)$ or $\Max(D')$ is finite, then from Corollary \ref{pid} it follows that $\Max(D)$ and $\Max(D')$ have the same (finite) cardinality.  Alternatively, if $\Max(D)$ and $\Max(D')$ are infinite and one assumes the statement (P), then since $2^{2^{|\Max(D)|}} = 2^{2^{|\Max(D')|}}$ by Corollary \ref{dedcard}, one has $|\Max(D)| = |\Max(D')|$.  Finally, since one can prove in ZFC that there is a PID $D$ with $|\Max(D)|$ equal to any given infinite cardinal $\aleph_a$ (namely, $D = k[T]$, where $k$ is an algebraic closure of $\QQ(\mathbf{X})$, where $|\mathbf{X}| = \aleph_a$), it follows that the implication $(4) \Rightarrow (1)$ is equivalent in ZFC to the statement (P) and therefore can neither be proved nor disproved from ZFC if ZFC is consistent.
\end{proof}

\begin{problem}
In Corollary \ref{equiv} does  (3) imply (2)?  Does (4) imply (3)?
\end{problem}

\section{The lattice of finite type semistar operations}

A semistar operation $\star$ on a domain $D$ is said to be of {\it finite type} if $I^\star = \bigcup\{J^*: J \subseteq I \textup{\ is finitely generated}\}$ for all $I \in \F(D)$.  By \cite[Proposition 13.6]{ell} the set $\SStar_f(D)$ of all finite type star operations on $D$ is a sublattice of $\SStar(D)$ that is closed under arbitrary suprema and finite infima; in particular it is a complete lattice.  In this section we construct the lattice $\SStar_f(D)$ for any  Dedekind domain $D$. As one might suspect this is a much smaller lattice than $\SStar(D)$ and is much easier to construct.

A closure operation $\star$ on a poset $S$ is said to be {\it finitary} if $(\bigvee \Delta)^\star = \bigvee (\Delta^\star)$ for all directed subsets $\Delta$ of $S$ for which $\bigvee \Delta$ exists.   A semistar operation on a domain $D$ is of finite type if and only if it is finitary as a closure operation on the lattice $\F(D)$. 

Let $M$ be a near multiplicative lattice.  By \cite[Proposition 5.3]{ell} the subset $\N_f(M)$ of $\N(M)$ consisting of all finitary nuclei on $M$ is closed under arbitrary suprema and in particular is a complete lattice.  Let $$\R(M) = \{x \in M: x^2 = x \mbox{ and } x \geq 1\}.$$  Then $\R(M)$ is also a near multiplicative lattice by \cite[Proposition 9.1]{ell}.  For any $a \in M$ we let $d_a: M \longrightarrow M$ be given by $x^{d_a} = a x$ for all $x \in M$.  Then $d_a$ is a finitary nucleus on $M$, and by \cite[Proposition 10.4]{ell} the map $d_-: \R(M) \longrightarrow \N_f(M)$ is a supremum-preserving poset embedding.

\begin{example}
If $D$ is an integral domain, then $\R(\F(D))$ is the lattice $\mathcal{O}(D)$ of all overrings of $D$ and $\N_f(\F(D))$ is the lattice $\SStar_f(D)$ of all finite type semistar operations on $D$.  One has $I^{d_A} = AI$ for any overring $A$ of $D$ and any $I \in \F(D)$. 
\end{example}

\begin{proposition}
Let $D$ be a Dedekind domain.
\begin{enumerate}
\item The map $\rho: 2^{\operatorname{Max}(D)} \longrightarrow \mathcal{O}(D)$ acting by
$\rho: X \longmapsto \bigcap_{\ppp \in X} D_\ppp$ is an anti-isomorphism of posets.
\item The map $d_- : \mathcal{O}(D) = \R(\F(D)) \longrightarrow \SStar_f(D) = \N_f(\F(D))$ is an isomorphism of posets.
\item The map $2^{\operatorname{Max}(D)} \longrightarrow \SStar_f(D)$ acting by $X \longmapsto d_{\rho(X)}$ is an anti-isomorphism of posets.
\end{enumerate}
\end{proposition}

\begin{proof}
By \cite[Proposition 3.14]{hei}, if $A$ is an overring of a Dedekind domain $D$, then there exists a unique subset $X$ of $\Max(D)$ such that $A =  \bigcap_{\ppp \in X} D_\ppp$.    Statement (1) follows.  Statement (2), then, follow from \cite[Example 10.8]{ell}, and statement (3) follows from statements (1) and (2).
\end{proof}

\begin{corollary}
One has $|\SStar_f(D)| = 2^{|\Max(D)|}$ for any Dedekind domain $D$.
\end{corollary}

\end{document}